\documentclass[10pt]{amsart}
\usepackage{graphicx,amsmath,amssymb}
\newtheorem{thm}{Theorem}[section]

\newtheorem{lem}[thm]{Lemma}
\newtheorem{prop}[thm]{Proposition}
\theoremstyle{definition}

\theoremstyle{remark}

\numberwithin{equation}{section}

\begin{document}

\title[Strichartz estimates for the magnetic Schr\"odinger equation]
    {Strichartz estimates for the magnetic Schr\"odinger equation with potentials $V$ of critical decay}%
\author{Seonghak Kim and Youngwoo Koh}%

\address{Department of Mathematics, Kyungpook National University, Daegu 41566, Republic of Korea}
\email{kimseo14@gmail.com}%

\address{Department of Mathematics Education, Kongju National University, Kongju 32588, Republic of Korea}
\email{ywkoh@kongju.ac.kr}%


\begin{abstract}
We study the Strichartz estimates for the magnetic Schr\"odinger equation in dimension $n\geq3$.
More specifically, for all Schr\"odinger admissible pairs $(r,q)$, we establish the estimate
    $$
    \|e^{itH}f\|_{L^{q}_{t}(\mathbb{R}; L^{r}_{x}(\mathbb{R}^n))} \leq C_{n,r,q,H} \|f\|_{L^2(\mathbb{R}^n)}
    $$
when the operator $H= -\Delta_A +V$ satisfies suitable conditions.
In the purely electric case $A\equiv0$,
we extend the class of potentials $V$ to the Fefferman-Phong class.
In doing so, we apply a weighted estimate for the Schr\"odinger equation developed by Ruiz and Vega.
Moreover, for the endpoint estimate of the magnetic case in $\mathbb{R}^3$,
we  investigate an equivalence
    $$
    \| H^{\frac{1}{4}} f \|_{L^r(\mathbb{R}^3)} \approx C_{H,r} \big\| (-\Delta)^{\frac{1}{4}} f \big\|_{L^r(\mathbb{R}^3)}
    $$
and find sufficient conditions on $H$ and $r$ for which the equivalence holds.
\end{abstract}
\maketitle


\section{Introduction}

Consider the Cauchy problem of the magnetic Schr\"odinger equation in $\mathbb{R}^{n+1}$ $(n\ge 3)$:
    \begin{equation}\label{mS_eq}
    \begin{cases}
    i\partial_{t}u - Hu=0, \quad (x,t)\in \mathbb{R}^{n} \times \mathbb{R},\\
    u(x,0)=f(x), \quad\, f\in\mathcal S.
    \end{cases}
    \end{equation}
Here, $\mathcal{S}$ is the Schwartz class, and $H$ is the electromagnetic Schr\"odinger operator
    $$
    H = -\nabla_A^2 +V(x), \quad \nabla_A = \nabla - iA(x),
    $$
where $A= (A^1,A^2,\cdots,A^n): \mathbb{R}^{n} \rightarrow \mathbb{R}^{n}$ and  $V:\mathbb{R}^{n} \rightarrow \mathbb{R}$.
The \textit{magnetic field} $B$ is defined by
    $$
    B= DA - (DA)^T\in \mathcal{M}_{n \times n},
    $$
where $(DA)_{ij}= \partial_{x_i} A^j$, $(DA)^T$ denotes the transpose of $DA$, and $\mathcal{M}_{n \times n}$ is the space of $n\times n$ real matrices.
In dimension $n=3$, $B$ is  determined by the cross product with the vector field $\mathrm{curl} A$:
    $$
    Bv = \mathrm{curl} A \times v\quad (v \in \mathbb{R}^3).
    $$

In this paper, we consider the Strichartz type estimate
    \begin{equation}\label{eM result}
    \|u\|_{L^{q}_{t}(\mathbb{R}; L^{r}_{x}(\mathbb{R}^n))} \leq C_{n,r,q,H} \|f\|_{L^2(\mathbb{R}^n)},
    \end{equation}
where $u=e^{itH}f$ is the solution to problem \eqref{mS_eq} with solution operator $e^{itH}$,
and study some conditions on $A$, $V$ and pairs $(r,q)$ for which the estimate holds.

For the unperturbed case  of \eqref{mS_eq} that $A \equiv 0$ and $V \equiv 0$,
Strichartz \cite{St} proved the inequality
    $$
    \|e^{it\Delta}f\|_{L^{\frac{2(n+2)}{n}}(\mathbb{R}^{n+1})} \leq C_n \|f\|_{L^2(\mathbb{R}^{n})},
    $$
where $e^{it\Delta}$ is the solution operator given by
    $$
    e^{it\Delta}f(x)= \frac{1}{(2\pi)^n} \int_{\mathbb{R}^n} e^{ix\cdot\xi + it|\xi|^{2}}\widehat{f}(\xi)d\xi.
    $$
Later, Keel and Tao \cite{KT} generalized this inequality to the following:
    \begin{equation}\label{KT_result}
    \|e^{it\Delta}f\|_{L^{q}_{t}(\mathbb{R}; L^{r}_{x}(\mathbb{R}^n))} \leq C_{n,r,q} \|f\|_{L^2(\mathbb{R}^n)}
    \end{equation}
holds if and only if $(r,q)$ is a Schr\"odinger admissible pair; that is,
$r,q \geq 2$, $(r,q)\neq (\infty,2)$ and $\frac{n}{r} + \frac{2}{q} = \frac{n}{2}$.

In the purely electric case of (\ref{mS_eq}) that $A \equiv 0$, the decay $|V(x)| \sim 1/|x|^2$ has been known to be critical for the validity of the Strichartz estimate.
It was shown by Goldberg, Vega and Visciglia \cite{GVV} that for each $\epsilon>0$,  there is a counterexample of $V=V_\epsilon$ with $|V(x)| \sim |x|^{-2+\epsilon}$ for $|x|\gg 1$ such that the estimate fails to hold.
In a postive direction, Rodnianski and Schlag \cite{RS} proved
    \begin{equation}\label{electiric_result}
    \|u\|_{L^{q}_{t}(\mathbb{R}; L^{r}_{x}(\mathbb{R}^n))} \leq C_{n,r,q,V} \|f\|_{L^2(\mathbb{R}^n)}
    \end{equation}
for non-endpoint admissible pairs $(r,q)$ (i.e., $q>2$) with almost critical decay $|V(x)| \lesssim 1/(1+|x|)^{2+\epsilon}$.
On the other hand, Burq, Planchon, Stalker and Tahvildar-Zadeh \cite{BPST} established \eqref{electiric_result} for critical decay $|V(x)| \lesssim 1/|x|^2$ with some technical conditions on $V$ but the endpoint case included.
Other than these, there have been many related positive results; see, e.g., \cite{GS}, \cite{Go}, \cite{BRV}, \cite{FV} and \cite{BG}.

In regard to the purely electric case, the following is the first main result of this paper whose proof is given in section \ref{sec_A=0}.

\begin{thm}\label{main_thm_1}
Let $n\geq3$ and $A \equiv 0$.
Then there exists a constant $c_n>0$, depending only on $n$, such that
for any $V \in \mathcal{F}^p$ $(\frac{n-1}{2}<p<\frac{n}{2})$ satisfying
    \begin{equation}\label{Cond_Decay_V}
    \|V\|_{\mathcal{F}^p} \leq c_n,
    \end{equation}
estimate \eqref{electiric_result} holds for all $\frac{n}{r} + \frac{2}{q} = \frac{n}{2}$ and $q>2$.
Moreover, if $V\in L^{\frac{n}{2}}$ in addition,
then  \eqref{electiric_result} holds for the endpoint case $(r,q)=(\frac{2n}{n-2},2)$.
\end{thm}

Here, $\mathcal{F}^p$ is the Fefferman-Phong class with norm
    $$
    \|V\|_{\mathcal{F}^p}
    = \sup_{r>0,\,x_0\in\mathbb{R}^n} r^2 \Big( \frac{1}{r^n} \int_{B_r(x_0)}|V(x)|^p dx \Big)^{\frac{1}{p}} < \infty,
    $$
which is closed under translation.
From the definition of $\mathcal{F}^p$,
we directly get $L^{\frac{n}{2},\infty} \subset \mathcal{F}^p$ for all $p<\frac{n}{2}$.
Thus the class $\mathcal{F}^p$ $(p<\frac{n}{2})$ clearly contains the potentials of critical decay $|V(x)| \lesssim 1/|x|^2$.
Moreover, $\mathcal{F}^p$ $(p<\frac{n}{2})$ is strictly larger than $L^{\frac{n}{2},\infty}$.
For instance, if the potential function
    \begin{equation}\label{example_Fp}
    V(x)=\phi\big(\frac{x}{|x|}\big) |x|^{-2},\quad \phi\in L^p(S^{n-1}),\quad \frac{n-1}{2}<p<\frac{n}{2},
    \end{equation}
then $V$ need not belong to $L^{\frac{n}{2},\infty}$, but $V\in \mathcal{F}^p$.

According to Theorem \ref{main_thm_1}, for the non-endpoint case,
we do not need any other conditions on  $V$ but its quantitative bound \eqref{Cond_Decay_V},
so that we can extend and much simplify the known results for potentials $|V(x)| \sim 1/|x|^2$
(e.g., $\phi\in L^{\infty}(S^{n-1})$ in \eqref{example_Fp}), mentioned above.
To prove this, we use a weighted estimate developed by Ruiz and Vega \cite{RV2}.
We remark that our proof follows an approach different from those used in the previous works.

Unfortunately, for the endpoint case, we need an additional condition that $V\in L^{\frac{n}{2}}$.
Although $L^{\frac{n}{2}}$ dose not contain the potentials of critical decay,
it still includes those of \emph{almost} critical decay $|V(x)| \lesssim \phi(\frac{x}{|x|})\min (|x|^{-(2+\epsilon)}, |x|^{-(2-\epsilon)})$, $\phi\in L^{\frac{n}{2}}(S^{n-1})$.

In case of dimension $n=3$, we can find a specific bound for $V$, which plays the role of $c_n$ in Theorem \ref{main_thm_1}.
We state this as the second result of the paper.

\begin{thm}\label{main_thm_1'}
If $n=3$, $A \equiv 0$ and $\|V\|_{L^{3/2}} < 2\pi^{1/3}$,
then estimate \eqref{electiric_result} holds for all $\frac{3}{r} + \frac{2}{q} = \frac{3}{2}$ and $q\geq 2$.
\end{thm}

To prove this, we use the best constant of the Stein-Tomas restriction theorem in $\mathbb{R}^3$, obtained by Foschi \cite{F},
and apply it to an argument in Ruiz and Vega \cite{RV2}.\\

Next, we consider the general (magnetic) case that $A$ or $V$ can be different from zero.
In this case, the Coulomb decay $|A(x)| \sim 1/|x|$ seems  critical.
(In \cite{FG}, there is a counterexample for $n\geq3$. The case $n=2$ is still open.)
In an early work of Stefanov \cite{Ste},   estimate (\ref{eM result}) for $n\ge 3$ was proved, that is,
    \begin{equation}\label{magnetic_result}
    \|e^{itH}f\|_{L^{q}_{t}(\mathbb{R}; L^{r}_{x}(\mathbb{R}^n))} \leq C_{n,r,q,H} \|f\|_{L^2(\mathbb{R}^n)}
    \end{equation}
for all Schr\"odinger admissible pairs $(r,q)$ under some smallness assumptions on the potentials $A$ and $V$.
Later, for potentials of almost critical decay $|A(x)| \lesssim 1/|x|^{1+\epsilon}$ and $|V(x)| \lesssim 1/|x|^{2+\epsilon}$ $(|x|\gg1)$,
D'Ancona, Fanelli, Vega and Visciglia \cite{DFVV} established \eqref{magnetic_result}
for all Schr\"odinger admissible pairs $(r,q)$ in $n\ge 3$, except the endpoint case $(n,r,q) = (3,6,2)$,
under some technical conditions on $A$ and $V$.
Also, there have been many related positive results; see, e.g., \cite{GST}, \cite{EGS2}, \cite{DF}, \cite{MMT}, \cite{EGS1}, \cite{Go2} and \cite{FFFP}.
Despite all these results, there has been no known positive result on the estimate in case of potentials $A$ of critical decay even in the case $V\equiv 0$. 

Regarding the general case, we state the last result of the paper whose proof is provided in section \ref{sec_V=0}.

\begin{thm}\label{main_thm_2}
Let $n\ge 3$, $A,V \in C^1_{loc}(\mathbb{R}^n\backslash \{0\})$ and $\epsilon>0$.
Assume that the operator $\Delta_A = -(\nabla -iA)^2$ and $H= \Delta_A +V$ are self-adjoint and positive on $L^2$
and that
    \begin{equation}\label{Kato_class}
    \| V_- \|_K = \sup_{x \in \mathbb{R}^n} \int \frac{|V_-(y)|}{|x-y|^{n-2}} dy
    ~< \frac{\pi^{n/2}}{\Gamma(\frac{n}{2}-1)} .
    \end{equation}
Assume also that there is a constant $C_{\epsilon}>0$
such that $A$ and $V$ satisfy the almost critical decay condition
    \begin{equation}\label{Cond_2'}
    \begin{aligned}
    |A(x)|^2 + |V(x)|
    \leq C_{\epsilon} \min\big( \frac{1}{|x|^{2-\epsilon}}, \frac{1}{|x|^{2+\epsilon}}\big) ,
    \end{aligned}
    \end{equation}
and the Coulomb gauge condition
    \begin{equation}\label{Cond_Coulomb}
    \nabla\cdot A =0.
    \end{equation}
Lastly, for the \textit{trapping component} of $B$ as
$B_\tau (x) = (x/|x|) \cdot B(x)$,
assume that
	\begin{equation}\label{Cond_B_n=3}
	\int_0^\infty \sup_{|x|=r} |x|^3 \big| B_\tau (x) \big|^2 dr
	+ \int_0^\infty \sup_{|x|=r} |x|^2 \big| \big( \partial_r V (x) \big)_{+} \big| dr
	< \frac{1}{M}\quad\mbox{if}\quad n=3,
	\end{equation}
for some $M>0$, and that
    \begin{equation}\label{Cond_B_n>3}
	\Big\| |x|^2 B_\tau (x) \Big\|_{L^{\infty}}^2
	+ 2 \Big\| |x|^3 \big(\partial_r V (x) \big)_{+} \Big\|_{L^{\infty}}
	< \frac{2(n-1)(n-3)}{3}\quad\mbox{if}\quad n\geq4.
	\end{equation}
Only for $n=3$, we also assume the boundness of the imagnary power of $H$: 
    \begin{equation}\label{Cond_BMO}
    \| H^{iy} \|_{BMO \rightarrow BMO_H} \leq C (1+|y|)^{3/2}.
    \end{equation}
Then we have
    \begin{equation}\label{result_Thm2}
    \|e^{itH}f\|_{L^{q}_{t}(\mathbb{R}; L^{r}_{x}(\mathbb{R}^n))} \leq C_{n,r,q,H,\epsilon} \|f\|_{L^2(\mathbb{R}^n)},
    \quad \frac{n}{r} + \frac{2}{q} = \frac{n}{2} \quad\mbox{and}\quad q \geq 2 .
    \end{equation}
\end{thm}

Note that this result covers the endpoint case $(n,r,q)=(3,6,2)$; but the conclusions for the other cases are the same as in \cite{DFVV}.
Here, $V_{\pm}$ denote the positive and negative parts of $V$, respectively; that is, $V_+=\max\{V,0\}$ and $V_-=\max\{-V,0\}$.
Also, we say that a function $V$ is of Kato class if
    $$
    \| V \|_K := \sup_{x \in \mathbb{R}^n} \int \frac{|V(y)|}{|x-y|^{n-2}} dy ~< \infty,
    $$
and $\Gamma$ in \eqref{Kato_class} is the gamma function,  defined by
$\Gamma(\alpha)= \int_{0}^\infty x^{\alpha-1} e^{-x} dx$.
The last condition \eqref{Cond_BMO} for $n=3$ may seem a bit technical but not be artificial.
For instance, by Lemma 6.1 in \cite{DDSY},
we know that $\| H^{iy} \|_{L^{\infty} \rightarrow BMO_H} \leq C (1+|y|)^{3/2}$ using only \eqref{Kato_class}.
Also, there are many known sufficient conditions to extend such an estimate to $BMO\rightarrow BMO$, like the translation invariant operator (See \cite{P}). For the definition and some basic properties of $BMO_H$ space, see section \ref{sec_CKN_ineq}.

The rest of the paper is organized as follows.
In section \ref{sec_A=0}, we prove Theorems \ref{main_thm_1} and \ref{main_thm_1'}. An equivalence of norms regarding $H$ and $-\Delta$ in $\mathbb{R}^3$ is investigated in section \ref{sec_CKN_ineq}. Lastly, in section \ref{sec_V=0}, Theorem \ref{main_thm_2} is proved.




\section{The case $A\equiv0$: Proof of Theorem \ref{main_thm_1} and \ref{main_thm_1'}}\label{sec_A=0}

In this section,  the proof of Theorems \ref{main_thm_1} and \ref{main_thm_1'} is provided.
Let $n\geq3$, and consider the purely electric Schr\"odinger equation in $\mathbb{R}^{n+1}$:
    \begin{equation}\label{mS_eq_A=0}
    \begin{cases}
    i\partial_{t}u +\Delta u= V(x)u, \quad (x,t)\in \mathbb{R}^{n} \times \mathbb{R},\\
    u(x,0)=f(x), \quad\qquad\, f\in\mathcal S.
    \end{cases}
    \end{equation}
By Duhamel's principle, we have a formal solution to problem \eqref{mS_eq_A=0} given by
    $$
    u(x,t)= e^{itH}f(x) =e^{it\Delta}f(x) - i \int_{0}^{t}e^{i(t-s)\Delta} V(x) e^{isH}f ds .
    $$
From the standard Strichartz estimate \eqref{KT_result}, there holds
    $$
    \begin{aligned}
    \|e^{itH}f\|_{L^{q}_{t}(\mathbb{R}; L^{r}_{x}(\mathbb{R}^n))}
    &\leq C_{n,r,q} \|f\|_{L^2(\mathbb{R}^n)}
        + \Big\| \int_{0}^{t}e^{i(t-s)\Delta} V(x) e^{isH}f ds \Big\|_{L^{q}_{t}(\mathbb{R}; L^{r}_{x}(\mathbb{R}^n))}
    \end{aligned}
    $$
for all Schr\"odinger admissible pairs $(r,q)$.
Thus it is enough to show that
    \begin{equation}\label{A=0_reduce_1}
    \Big\| \int_{0}^{t}e^{i(t-s)\Delta} V(x) e^{isH}f ds \Big\|_{L^{q}_{t}(\mathbb{R}; L^{r}_{x}(\mathbb{R}^n))}
    \leq C_{n,r,q,V} \| f \|_{L^2(\mathbb{R}^n)} .
    \end{equation}
for all Schr\"odinger admissible pairs $(r,q)$.

By the duality argument, estimate \eqref{A=0_reduce_1} is equivalent to
    $$
    \int_{\mathbb{R}} \int_{0}^{t} \Big\langle e^{i(t-s)\Delta} \big( V(x) e^{isH}f \big), G(\cdot,t) \Big\rangle_{L^2_x} ds dt
    \leq C \| f \|_{L^2(\mathbb{R}^n)} \| G \|_{L^{q'}_{t}(\mathbb{R}; L^{r'}_{x}(\mathbb{R}^n))}.
    $$
Now, we consider the left-hand side of this inequality.
Commuting the operator and integration, we have
    $$
    \begin{aligned}
    &\int_{\mathbb{R}} \int_{0}^{t} \Big\langle e^{i(t-s)\Delta} \big( V(x) e^{isH}f \big), G(\cdot,t) \Big\rangle_{L^2_x} ds dt \\
    &\quad\quad= \int_{\mathbb{R}} \int_{0}^{t} \Big\langle V(x) e^{isH}f, e^{-i(t-s)\Delta} G(\cdot,t) \Big\rangle_{L^2_x} ds dt \\
    &\quad\quad= \int_{\mathbb{R}} \Big\langle V(x) e^{isH}f, \int_{s}^{\infty} e^{-i(t-s)\Delta} G(\cdot,t)dt \Big\rangle_{L^2_x} ds.
    \end{aligned}
    $$
By H\"older's inequality, we have
    $$
    \begin{aligned}
    &\int_{\mathbb{R}} \Big\langle V(x) e^{isH}f, \int_{s}^{\infty} e^{-i(t-s)\Delta} G(\cdot,t)dt \Big\rangle_{L^2_x} ds \\
    &\qquad\leq \big\| e^{isH}f \big\|_{L^2_{x,s} (|V|)}
        \Big\| \int_{s}^{\infty} e^{-i(t-s)\Delta} G (\cdot,t) dt \Big\|_{L^2_{x,s} (|V|)}.
    \end{aligned}
    $$
Thanks to  \cite[Theorem 3]{RV2}, for any $\frac{n-1}{2}<p<\frac{n}{2}$, we have
    \begin{equation}\label{RV2_result}
    \big\| e^{itH}f \big\|_{L^2_{x,t} (|V|)}
    \leq C_{n} \|V\|_{\mathcal{F}^p}^{\frac{1}{2}} \| f \|_{L^2(\mathbb{R}^n)}
    \end{equation}
if condition \eqref{Cond_Decay_V} holds for some suitable constant $c_n$.
More specifically, by Propositions 2.3 and 4.2 in \cite{RV2}, we have
    $$
    \begin{aligned}
    \big\| e^{itH}f \big\|_{L^2_{x,t} (|V|)}
    &\leq \big\| e^{it\Delta}f \big\|_{L^2_{x,t} (|V|)}
        + \Big\| \int_{0}^{t}e^{i(t-s)\Delta} V(x) e^{isH}f ds \Big\|_{L^2_{x,t} (|V|)} \\
    &\leq C_1 \|V\|_{\mathcal{F}^p}^{\frac{1}{2}} \| f \|_{L^2}
        + C_2 \|V\|_{\mathcal{F}^p} \big\| V(x) e^{itH}f \big\|_{L^2_{x,t}(|V|^{-1})} \\
    &= C_1 \|V\|_{\mathcal{F}^p}^{\frac{1}{2}} \| f \|_{L^2}
        + C_2 \|V\|_{\mathcal{F}^p} \| e^{itH}f \|_{L^2_{x,t}(|V|)}.
    \end{aligned}
    $$
Thus, if $\|V\|_{\mathcal{F}^p} \leq 1/(2C_2)=:c_n$, we get
    $$
    \big\| e^{itH}f \big\|_{L^2_{x,t} (|V|)}
    \leq C_1 \|V\|_{\mathcal{F}^p}^{\frac{1}{2}} \| f \|_{L^2} + \frac{1}{2}\| e^{itH}f \|_{L^2_{x,t}(|V|)},
    $$
and this implies \eqref{RV2_result} by setting $C_n=2C_1$.
As a result, we can reduce \eqref{A=0_reduce_1} to
    \begin{equation}\label{CK_goal}
    \Big\| \int_{t}^{\infty} e^{i(t-s)\Delta} G (\cdot,s) ds \Big\|_{L^2_{x,t} (|V|)}
    \leq C_{n,r,q,V} \| G \|_{L^{q'}_{t}(\mathbb{R}; L^{r'}_{x}(\mathbb{R}^n))}.
    \end{equation}

It now remains to establish (\ref{CK_goal}).
First, from \cite[Proposition 2.3]{RV2} and the duality of Keel-Tao's result \eqref{KT_result},
we know
    \begin{equation}\label{CK_reduce}
    \begin{aligned}
    \Big\| \int_{\mathbb{R}} e^{i(t-s)\Delta} G (\cdot,s) ds \Big\|_{L^2_{x,t}(|V|)}
    &=\Big\| e^{it\Delta} \int_{\mathbb{R}} e^{-is\Delta} G (\cdot,s) ds \Big\|_{L^2_{x,t}(|V|)} \\
    &\leq C_n \|V\|_{\mathcal{F}^p}^{\frac{1}{2}} \Big\| \int_{\mathbb{R}} e^{-is\Delta} G (\cdot,s) ds \Big\|_{L^2_x} \\
    &\leq C_{n,r,q} \|V\|_{\mathcal{F}^p}^{\frac{1}{2}} \| G \|_{L^{q'}_t L^{r'}_x}
    \end{aligned}
    \end{equation}
for all Schr\"odinger admissible pair $(r,q)$.
In turn, \eqref{CK_reduce} implies
    \begin{equation}\label{CK_reduce_2}
    \Big\| \int_{-\infty}^{t} e^{i(t-s)\Delta} G (\cdot,s) ds \Big\|_{L^2_{x,t} (|V|)}
    \leq C_{n,r,q} \|V\|_{\mathcal{F}^p}^{\frac{1}{2}} \| G \|_{L^{q'}_{t}(\mathbb{R}; L^{r'}_{x}(\mathbb{R}^n))}
    \end{equation}
by the Christ-Kiselev lemma \cite{CK} for $q>2$.
Combining \eqref{CK_reduce} with \eqref{CK_reduce_2}, we directly get \eqref{CK_goal} for $q>2$.
Next, for the endpoint case $(r,q)=(\frac{2n}{n-2},2)$, we have
    \begin{equation}\label{CK_reduce_3}
    \begin{aligned}
    \Big\| \int_{-\infty}^{t} e^{i(t-s)\Delta} G (\cdot,s) ds \Big\|_{L^2_{x,t} (|V|)}
    &\leq \|V\|_{L^{\frac{n}{2}}_x}^{\frac{1}{2}} \Big\| \int_{-\infty}^{t} e^{i(t-s)\Delta} G (\cdot,s) ds \Big\|_{L^{2}_{t}(\mathbb{R}; L^{\frac{2n}{n-2}}_{x}(\mathbb{R}^n))} \\
    &\leq C_n \|V\|_{L^{\frac{n}{2}}_x}^{\frac{1}{2}} \| G \|_{L^{2}_{t}(\mathbb{R}; L^{\frac{2n}{n+2}}_{x}(\mathbb{R}^n))}
    \end{aligned}
    \end{equation}
from H\"older's inequality in $x$ with the inhomogeneous Strichartz estimates by Keel-Tao.
Observe now that \eqref{CK_reduce_3} implies \eqref{CK_goal} when $q=2$ with the assumption $V\in L^{\frac{n}{2}}_x$.

The proof of Theorem \ref{main_thm_1} is now complete.\\

Now, we will find a suitable constant in Theorem \ref{main_thm_1'}.
For this, we refine estimate \eqref{RV2_result} based on an argument in \cite{RV2}.
We recall the Fourier transform in $\mathbb{R}^n$, defined by
    $$
    \widehat{f}(\xi)=\int_{\mathbb{R}^n} e^{ix\cdot\xi} f(x)dx,
    $$
and its basic properties
    $$
    f(x) = \frac{1}{(2\pi)^{n}} \int_{\mathbb{R}^n} e^{ix\cdot\xi} \widehat{f}(\xi)d\xi
    \quad\mbox{and}\quad
    \| f\|_{L^2(\mathbb{R}^n)} = \frac{1}{(2\pi)^{\frac{n}{2}}} \big\| \widehat{f} \big\|_{L^2(\mathbb{R}^n)} .
    $$
Thus, we can express $e^{it\Delta}f$ using the polar coordinates with $r^2=\lambda$ as follows:
    $$
    \begin{aligned}
    e^{it\Delta}f
    &= \frac{1}{(2\pi)^{n}} \int_0^\infty e^{itr^2} \int_{S^{n-1}_r} e^{ix\cdot\xi} \widehat{f}(\xi) d\sigma_r(\xi) dr \\
    &= \frac{1}{2(2\pi)^{n}} \int_0^\infty e^{it\lambda} \int_{S^{n-1}_{\sqrt{\lambda}}} e^{ix\cdot\xi} \widehat{f}(\xi) d\sigma_{\sqrt{\lambda}}(\xi) \lambda^{-\frac{1}{2}} d\lambda.
    \end{aligned}
    $$
Take $F$ as
    $$
    F(\lambda)
    = \int_{S^{n-1}_{\sqrt{\lambda}}} e^{ix\cdot\xi} \widehat{f}(\xi) d\sigma_{\sqrt{\lambda}}(\xi) \lambda^{-\frac{1}{2}}
    $$
if $\lambda\geq0$ and $F(\lambda)=0$ if $\lambda<0$.
Then, by Plancherel's theorem in $t$, we get
    $$
    \begin{aligned}
    \big\| e^{it\Delta}f \big\|_{L^2_{x,t} (|V|)}^2
    &= \frac{2\pi}{4(2\pi)^{2n}} \int_{\mathbb{R}^n} \Big( \int_{\mathbb{R}} |F(\lambda)|^2 d\lambda \Big) |V(x)|dx \\
    &= \frac{\pi}{2(2\pi)^{2n}} \int_{\mathbb{R}^n} \Big( \int_{0}^{\infty} \Big| \int_{S^{n-1}_{\sqrt{\lambda}}} e^{ix\cdot\xi} \widehat{f}(\xi) d\sigma_{\sqrt{\lambda}}(\xi) \Big|^2 \lambda^{-1} d\lambda \Big) |V(x)|dx \\
    &= \frac{\pi}{(2\pi)^{2n}} \int_{0}^{\infty} \Big( \int_{\mathbb{R}^n} \Big| \int_{S^{n-1}_r} e^{ix\cdot\xi} \widehat{f}(\xi) d\sigma_r(\xi) \Big|^2 |V(x)|dx \Big) r^{-1} dr .
    \end{aligned}
    $$
Now, we consider the $n=3$ case and apply the result on the best constant of the Stein-Tomas restriction theorem in $\mathbb{R}^3$ obtained by Foschi \cite{F}.
That is,
    $$
    \big\| \widehat{fd\sigma} \big\|_{L^4(\mathbb{R}^3)}
    \leq 2\pi \| f \|_{L^2(S^{2})}
    $$
where
    $$
    \widehat{fd\sigma}(x) = \int_{S^{n-1}} e^{-ix\cdot\xi}f(\xi)d\sigma(\xi).
    $$
Interpolating this with a trivial estimate
    $$
    \big\| \widehat{fd\sigma} \big\|_{L^{\infty}(\mathbb{R}^3)}
    \leq \| f \|_{L^1(S^{2})}
    \leq \sqrt{4\pi} \| f \|_{L^2(S^{2})},
    $$
we get
    $$
    \big\| \widehat{fd\sigma} \big\|_{L^6(\mathbb{R}^3)}
    \leq 2^{1/6} (2\pi)^{5/6} \| f \|_{L^2(S^{2})}.
    $$
By H\"older's inequality, we have
    $$
    \begin{aligned}
    \Big( \int_{\mathbb{R}^3} \Big| \int_{S^2} e^{ix\cdot\xi} \widehat{f}(\xi) d\sigma(\xi) \Big|^2 |V(x)|dx \Big)
    &\leq \Big\| \int_{S^2} e^{ix\cdot\xi} \widehat{f}(\xi) d\sigma(\xi) \Big\|_{L^6}^2 \| V \|_{L^{3/2}} \\
    &\leq 2^{1/3} (2\pi)^{5/3} \| V \|_{L^{3/2}} \| \widehat{f} \|_{L^2(S^{2})}^2.
    \end{aligned}
    $$
So we get
    $$
    \begin{aligned}
    \big\| e^{it\Delta}f \big\|_{L^2_{x,t} (|V|)}^2
    &\leq \frac{\pi}{(2\pi)^{6}} 2^{1/3} (2\pi)^{5/3} \| V \|_{L^{3/2}} \| \widehat{f} \|_{L^2}^2 \\
    &= \frac{1}{2\pi^{1/3}} \| V \|_{L^{3/2}} \| f \|_{L^2}^2 .
    \end{aligned}
    $$
By the argument as in the proof of Theorem \ref{main_thm_1}, we have
    $$
    \begin{aligned}
    \big\| e^{itH}f \big\|_{L^2_{x,t} (|V|)}
    &\leq \big\| e^{it\Delta}f \big\|_{L^2_{x,t} (|V|)}
        + \Big\| \int_{0}^{t}e^{i(t-s)\Delta} V(x) e^{isH}f ds \Big\|_{L^2_{x,t} (|V|)} \\
    &\leq \frac{1}{\sqrt{2}\pi^{1/6}} \| V \|_{L^{3/2}}^{\frac{1}{2}} \| f \|_{L^2}
        + \frac{1}{2\pi^{1/3}} \| V \|_{L^{3/2}} \big\| V(x) e^{itH}f \big\|_{L^2_{x,t}(|V|^{-1})} \\
    &= \frac{1}{\sqrt{2}\pi^{1/6}} \| V \|_{L^{3/2}}^{\frac{1}{2}} \| f \|_{L^2}
        + \frac{1}{2\pi^{1/3}} \| V \|_{L^{3/2}} \| e^{itH}f \|_{L^2_{x,t}(|V|)}.
    \end{aligned}
    $$
Thus, if $\| V \|_{L^{3/2}} < 2\pi^{1/3}$, then 
    \begin{equation}\label{const_WR}
    \big\| e^{itH}f \big\|_{L^2_{x,t} (|V|)}
    \leq C_V \| f \|_{L^2}.
    \end{equation}
Using \eqref{const_WR} instead of \eqref{RV2_result} in that  argument,
the proof of Theorem \ref{main_thm_1'} is complete.


\section{The equivalence of two norms involving $H$ and $-\Delta$ in $\mathbb{R}^3$}\label{sec_CKN_ineq}

In this section,
we investigate some conditions on $H$ and $p$ with which the equivalence
$$
\| H^{\frac{1}{4}} f \|_{L^p(\mathbb{R}^3)} \approx C_{H,p} \big\| (-\Delta)^{\frac{1}{4}} f \big\|_{L^p(\mathbb{R}^3)}
$$
holds.
This equivalence was studied in \cite{DFVV} and \cite{CD} that are of independent interest.
We now introduce such an equivalence in a form for $n=3$, which enables us to include the endpoint estimate also for that dimension.

\begin{prop}\label{Key_Prop}
	Given $A \in L^2_{loc}(\mathbb{R}^3; \mathbb{R}^3)$ and $V: \mathbb{R}^3 \rightarrow \mathbb{R}$ measurable,
	assume that the operators $\Delta_A = -(\nabla -iA)^2$ and $H= -\Delta_A +V$ are self-adjoint and positive on $L^2$ and that \eqref{Cond_BMO} holds.
	Moreover, assume that $V_+$ is of Kato class and that $A$ and $V$ satisfy \eqref{Kato_class} and
	\begin{equation}\label{upper_bdd_of_AV}
	|A(x)|^2 + |\nabla\cdot A(x)| + |V(x)|
	\leq C_0 \min\Big( \frac{1}{|x|^{2-\epsilon}}, \frac{1}{|x|^{2+\epsilon}} \Big)
	\end{equation}
	for some $0<\epsilon\leq2$ and $C_0>0$.
	Then the following estimates hold:
	\begin{equation}\label{key_leq_ineq}
	\| H^{\frac{1}{4}} f \|_{L^p} \leq C_{\epsilon,p} C_0 \| (-\Delta)^{\frac{1}{4}} f \|_{L^p},
	\quad 1<p \leq 6,
	\end{equation}
	\begin{equation}\label{key_geq_ineq}
	\| H^{\frac{1}{4}} f \|_{L^p} \geq C_p \| (-\Delta)^{\frac{1}{4}} f \|_{L^p},
	\quad \frac{4}{3}<p<4.
	\end{equation}
\end{prop}

In showing this, we only prove \eqref{key_leq_ineq} as
estimate \eqref{key_geq_ineq} is the same as \cite[Theorem 1.2]{DFVV}.
When $1<p<6$, estimate \eqref{key_leq_ineq} easily follows from the Sobolev embedding theorem.
However, in order to extend the range of $p$ up to $6$,
we need a precise estimate which depends on $\epsilon$ in \eqref{upper_bdd_of_AV}.
Towards this, we introduce a weighted Sobolev inequality as below.

\begin{lem}\label{lem_SW}
	\textbf{(Theorem 1(B) in \cite{SW2}).}
	Suppose $0<\alpha<n$, $1<p<q<\infty$ and $v_1(x)$ and $v_2(x)$ are nonnegative measurable functions on $\mathbb{R}^n$.
	Let $v_1(x)$ and $v_2(x)^{1-p'}$ satisfy the reverse doubling condition:
	there exist $\delta,\epsilon \in (0,1)$ such that
	$$
	\int_{\delta Q} v_1(x) \leq \epsilon\int_Q v_1(x)dx
	\quad\mbox{for all cubes}\quad Q\subset \mathbb{R}^n.
	$$
	Then the inequality
	$$
	\bigg( \int_{\mathbb{R}^n} |f(x)|^q v_1(x) dx \bigg)^{\frac{1}{q}}
	\leq C \bigg( \int_{\mathbb{R}^n} \big|(-\Delta)^{\alpha/2} f(x) \big|^p v_2(x) dx \bigg)^{\frac{1}{p}}
	$$
	holds if and only if
	$$
	|Q|^{\frac{\alpha}{n}-1} \Big( \int_Q v_1(x) dx \Big)^{\frac{1}{q}} \Big( \int_Q v_2(x)^{1-p'} dx \Big)^{\frac{1}{p'}} \leq C
	\quad\mbox{for all cubes}\quad Q\subset \mathbb{R}^n.
	$$
\end{lem}

From Lemma \ref{lem_SW}, we obtain a weighted estimate as follows.

\begin{lem}\label{lem_CKN_modi_2}
	Let $f$ be a $C_0^{\infty}(\mathbb{R}^3)$ function,
	and suppose that a nonnegative weight function $w$ satisfies
	\begin{equation}\label{CKN_pf_3}
	w(x) \leq \min\Big( \frac{1}{|x|^{2-\epsilon}}, \frac{1}{|x|^{2+\epsilon}} \Big)
	\end{equation}
	for some $0<\epsilon\leq2$. Then, for any $1<p \leq\frac{3}{2}$, we have
	$$
	\| f w \|_{L^p} \leq C_{\epsilon,p} \| \Delta f \|_{L^p}.
	$$
\end{lem}

\begin{proof}
	For all $1<p<\frac{3}{2}$, we directly get
	\begin{equation}\label{CKN_inter_1}
	\big\| \frac{1}{|x|^2} f \big\|_{L^p}
	\leq C \big\| \frac{1}{|x|^2} \big\|_{L^{\frac{3}{2},\infty}} \| f \|_{L^{\frac{3p}{3-2p},p}}
	\leq C \| \Delta f \|_{L^p}
	\end{equation}
	from H\"older's inequality in Lorentz spaces and the  Sobolev embedding theorem.
	For $p=\frac{3}{2}$, by H\"older's inequality, we get
	\begin{equation*}
	\bigg( \int_{\mathbb{R}^3} |f(x)|^{\frac{3}{2}} w(x)^{\frac{3}{2}} dx \bigg)^{\frac{2}{3}}
	\leq \bigg( \int_{\mathbb{R}^3} |f(x)|^{q} w(x)^{(1-\theta)q} dx \bigg)^{\frac{1}{q}}
	\bigg( \int_{\mathbb{R}^3} w(x)^{\frac{3q}{2q-3}\theta} dx \bigg)^{\frac{2q-3}{3q}}
	\end{equation*}
	for any $\frac{3}{2}< q<\infty$ and $0<\theta<1$.
	Taking $\theta = 1- \frac{3}{2q}$, we have
	\begin{equation*}
	\bigg( \int_{\mathbb{R}^3} |f(x)|^{\frac{3}{2}} w(x)^{\frac{3}{2}} dx \bigg)^{\frac{2}{3}}
	\leq C_{\epsilon,q} \bigg( \int_{\mathbb{R}^3} |f(x)|^{q} w(x)^{\frac{3}{2}} dx \bigg)^{\frac{1}{q}}
	\end{equation*}
	because of \eqref{CKN_pf_3}.
	Thus, using Lemma \ref{lem_SW} with $\alpha=2$, $(p,q)=(\frac{3}{2},q)$, $v_1(x)= w(x)^{\frac{3}{2}}$ and $v_2(x)\equiv 1$, we have
	\begin{equation}\label{CKN_inter_2}
	\bigg( \int_{\mathbb{R}^3} |f(x)|^{\frac{3}{2}} w(x)^{\frac{3}{2}} dx \bigg)^{\frac{2}{3}}
	\leq C_{\epsilon,q} \bigg( \int_{\mathbb{R}^3} |\Delta f(x) |^{\frac{3}{2}} w(x)^{\frac{3}{2}} dx \bigg)^{\frac{2}{3}}.
	\end{equation}
	Combining \eqref{CKN_inter_1} and \eqref{CKN_inter_2}, the proof is complete.
\end{proof}

Finally, we prove Proposition \ref{Key_Prop}.
We use Stein's interpolation theorem to the analytic family of operators $T_z= H^z \cdot (-\Delta)^{-z}$, where $H^z$ and $(-\Delta)^{-z}$ are defined by the spectral theory.
Denoting $z=x +iy$, we can decompose
$$
T_z = T_{x+iy} = H^{iy} H^x (-\Delta)^{-x} (-\Delta)^{-iy}.
$$
In fact, the operators $H^{iy}$ and $(-\Delta)^{-iy}$ are bounded according to the following result.

\begin{lem}\label{Im_est_operator}
	\textbf{(Proposition 2.2 in \cite{DFVV}).}
	Consider the self-adjoint and positive operators $-\Delta_A$ and $H= -\Delta_A +V$ on $L^2$.
	Assume that $A \in L^2_{loc}(\mathbb{R}^3; \mathbb{R}^3)$ and that the positive and negative parts $V_{\pm}$ of $V$ satisfy:
	$V_{+}$ is of Kato class
	and
	$$
	\| V_- \|_K < \frac{\pi^{3/2}}{\Gamma(1/2)}.
	$$
	Then for all $y\in \mathbb{R}$, the imaginary powers $H^{iy}$ satisfy the $(1,1)$ weak type estimate
	$$
	\| H^{iy} \|_{L^1\rightarrow L^{1,\infty}} \leq C (1+|y|)^{\frac{3}{2}}.
	$$
\end{lem}

Lemma \ref{Im_est_operator} follows from the pointwise estimate for the heat kernel $p_t(x,y)$ of the operator $e^{-tH}$ as
$$
\big| p_t(x,y) \big|
\leq \frac{(2t)^{-3/2}} {\pi^{3/2} - \Gamma(1/2)\|V_-\|_K} e^{-\frac{|x-y|^2}{8t}}.
$$
Regarding this  estimate, one may refer to some references  \cite{Sim, SW, CD, DSY}.

By Lemma \ref{Im_est_operator}, we get
\begin{equation}\label{x=0_est}
\| T_{iy} f\|_p \leq C (1+|y|)^3 \|f\|_p
\quad\mbox{for all}\quad 1<p<\infty.
\end{equation}
Then by \eqref{Cond_BMO}, we have
\begin{equation}\label{x=0_est_BMO}
\begin{aligned}
\| T_{iy} f \|_{BMO_H} 
&:= \big\| M^{\#}_{H} \big(H^{iy} (-\Delta)^{-iy} f \big) \big\|_{L^{\infty}} \\
&\leq C (1+|y|)^{\frac{3}{2}} \big\| (-\Delta)^{-iy} f \big\|_{BMO} 
\leq C (1+|y|)^3 \|f\|_{L^{\infty}},
\end{aligned}
\end{equation}
where
$$
M^{\#}_{H} f(x)
:= \sup_{r>0} \frac{1}{|B(x,r)|} \int_{B(x,r)} \big| f(y)-e^{-r^2H}f(y) \big|dy
<\infty.
$$

Next, consider the operator $T_{1+iy}$.
If
\begin{equation}\label{x=1_goal}
\| H (-\Delta)^{-1} f \|_{L^p} \leq C \| f \|_{L^p}
\quad\mbox{for all}\quad 1<p \leq \frac{3}{2},
\end{equation}
then by \eqref{x=0_est}, we get
\begin{equation}\label{x=1_est}
\| T_{1+iy} f \|_{L^p} \leq C \| f \|_{L^p}
\quad\mbox{for all}\quad 1<p \leq \frac{3}{2}.
\end{equation}
Taking $\widetilde{T}_z f := M^{\#}_{H} \big( T_{z}f\big)$ and applying \eqref{x=1_est} with a basic property\footnote{Some properties of the $BMO_L$ space can be found in \cite{DY}.}:
\begin{equation}\label{Lp_BMO_prop}
\| M^{\#}_{H} f \|_{L^p}
\leq C \|f\|_{L^p} \quad\mbox{for all}\quad 1<p\leq\infty,
\end{equation}
we have
\begin{equation}\label{x=1_est_2}
\| \widetilde{T}_{1+iy} f \|_{L^p} \leq C \| f \|_{L^p}
\quad\mbox{for all}\quad 1<p \leq \frac{3}{2}.
\end{equation}
So, applying Stein's interpolation theorem to \eqref{x=0_est_BMO} and \eqref{x=1_est_2}, we obtain
$$
\| \widetilde{T}_{1/4} f \|_{L^p} \leq C \| f \|_{L^p}
\quad\mbox{for all}\quad 1<p \leq 6,
$$
and using \eqref{Lp_BMO_prop} again, we have
$$
\| H^{1/4} f \|_{L^p} \leq C \| (-\Delta)^{1/4} f \|_{L^p}
\quad\mbox{for all}\quad 1<p \leq 6.
$$
Now, we handle the remaining part \eqref{x=1_goal}; that is, we wish to establish the estimate
$$
\| Hf \|_{L^p} \leq C \| \Delta f \|_{L^p}.
$$
For a Schwartz function $f$, we can write
\begin{equation}\label{expend_H}
Hf = -\Delta f +2iA\cdot \nabla f + (|A|^2 +i\nabla\cdot A +V)f.
\end{equation}
From H\"older's inequality in Lorentz spaces and the Sobolev embedding theorem, we get
$$
\| A \cdot \nabla f \|_{L^r}
\leq C \| A \|_{L^{3,\infty}} \| \nabla f \|_{L^{\frac{3r}{3-r},r}}
\leq C \| A \|_{L^{3,\infty}} \| \Delta f \|_{L^r}
$$
for all $1<r<3$.
On the other hand, applying Lemma \ref{lem_CKN_modi_2} to \eqref{upper_bdd_of_AV}, we get
$$
\big\| (|A|^2 +i\nabla\cdot A +V) f \big\|_{L^r}
\leq C C_0 \| \Delta f \|_{L^r}
$$
for all $1<r \leq\frac{3}{2}$.
Thus we have
$$
\| H f \|_{L^r} \leq C \| \Delta f \|_{L^r}
\quad\mbox{for all}\quad 1<r\leq \frac{3}{2},
$$
and this implies Proposition \ref{Key_Prop}.


\section{Proof of Theorem \ref{main_thm_2}}\label{sec_V=0}

In this final section, we prove Theorem \ref{main_thm_2}.
This part follows an argument in \cite{DFVV}.
Let $u$ be a solution to problem \eqref{mS_eq} of the magnetic Schr\"odinger equation in $\mathbb{R}^{n+1}$.
By \eqref{expend_H}, we can expand $H$ in \eqref{mS_eq}:
    $$
    H = -\Delta +2iA\cdot \nabla +|A|^2 +i\nabla\cdot A +V.
    $$
Thus, by Duhamel's principle and the Coulomb gauge condition \eqref{Cond_Coulomb},
we have a formal solution to \eqref{mS_eq} given by
    \begin{equation}\label{Duhamel_sol}
    u(x,t)= e^{itH}f(x) =e^{it\Delta}f(x) - i \int_{0}^{t}e^{i(t-s)\Delta} R(x,\nabla) e^{isH}f ds ,
    \end{equation}
where
    $$
    R(x,\nabla) = 2iA \cdot \nabla_{A} -|A|^2 + V.
    $$
From \cite{RV} and \cite{IK} (see also (3.4) in \cite{DFVV}),
it follows that for every admissible pair $(r,q)$,
    \begin{equation}\label{deri_inho_est}
    \Big\| |\nabla|^{\frac{1}{2}} \int_{0}^{t}e^{i(t-s)\Delta} F(\cdot,s) ds \Big\|_{L_t^q L_x^r}
    \leq C_{n,r,q} \sum_{j \in \mathbb{Z}} 2^{j/2} \| \chi_{C_j} F \|_{L^2_{x,t}},
    \end{equation}
where $C_j= \{ x: 2^j \leq |x| \leq 2^{j+1} \}$
and $\chi_{C_j}$ is the characteristic function of the set $C_j$.
Then, from \eqref{Duhamel_sol}, \eqref{KT_result} and \eqref{deri_inho_est}, we know
    $$
    \begin{aligned}
    \big\| |\nabla|^{\frac{1}{2}} u \big\|_{L^q_t L^r_x}
    &\leq \big\| |\nabla|^{\frac{1}{2}} e^{it\Delta}f \big\|_{L^q_t L^r_x} + \Big\| |\nabla|^{\frac{1}{2}} \int_{0}^{t}e^{i(t-s)\Delta} R(x,\nabla) e^{isH}f ds \Big\|_{L^q_t L^r_x}\\
    &\leq C_{n,r,q}\big\| |\nabla|^{1/2} f \big\|_{L^2_x} + C_{n,r,q} \sum_{j \in \mathbb{Z}} 2^{j/2} \Big\| \chi_{C_j} R(x,\nabla) e^{itH}f \Big\|_{L^2_{x,t}} .
    \end{aligned}
    $$
For the second term in the far right-hand side, we get
    $$
    \begin{aligned}
    &\Big\| \chi_{C_j} R(x,\nabla) e^{itH}f  \Big\|_{L^2_{x,t}} \\
    &\qquad\leq 2\Big\| \chi_{C_j} A\cdot \nabla_{A} e^{itH}f \Big\|_{L^2_{x,t}}
        + \Big\| \chi_{C_j} \big(|A|^2 +|V| \big) e^{itH}f \Big\|_{L^2_{x,t}}.
    \end{aligned}
    $$

Next, we will use a known result in \cite{FV},
which is a smoothing estimate for the magnetic Schr\"odinger equation.

\begin{lem}\label{Lamma_FV}
\textbf{(Theorems 1.9 and 1.10 in \cite{FV}).}
Assume $n\geq3$, $A$ and $V$ satisfy conditions \eqref{Cond_Coulomb}, \eqref{Cond_B_n=3} and \eqref{Cond_B_n>3}.
Then, for any solution $u$ to \eqref{mS_eq} with $f \in L^2$ and $-\Delta_{A} f \in L^2$,
the following estimate holds:
    $$
    \begin{aligned}
    &\sup_{R>0} \frac{1}{R} \int_{0}^{\infty} \int_{|x|\leq R} |\nabla_{A} u|^2 dxdt
    ~+~ \sup_{R>0} \frac{1}{R^2} \int_{0}^{\infty} \int_{|x|= R} |u|^2 d\sigma(x)dt \\
    &\qquad \leq C_{A} \| (-\Delta_{A})^\frac{1}{4} f \|_{L^2}^2.
    \end{aligned}
    $$
\end{lem}

From \eqref{Cond_2'} with Lemma \ref{Lamma_FV}, we have
    $$
    \begin{aligned}
    &\sum_{j \in \mathbb{Z}} 2^{j/2} \Big\| \chi_{C_j} A\cdot \nabla_{A} e^{itH}f \Big\|_{L^2_{x,t}} \\
    &\quad\leq \sum_{j \in \mathbb{Z}} 2^{j} \Big( \sup_{x\in C_j}|A| \Big)
        \Big( \frac{1}{2^{j+1}} \int_{0}^{\infty} \int_{|x|\leq 2^{j+1}} |\nabla_{A} u|^2 dxdt \Big)^{\frac{1}{2}}\\
    &\quad\leq \Big( \sum_{j \in \mathbb{Z}} 2^{j} \sup_{x\in C_j}|A| \Big)
        \Big( \sup_{R>0} \frac{1}{R} \int_{0}^{\infty} \int_{|x|\leq R} |\nabla_{A} u|^2 dxdt \Big)^{\frac{1}{2}}\\
    &\quad\leq C_{A,\epsilon} \big\| (-\Delta_{A})^\frac{1}{4} f \big\|_{L^2_x}
    \end{aligned}
    $$
and
    $$
    \begin{aligned}
    &\sum_{j \in \mathbb{Z}} 2^{j/2} \Big\| \chi_{C_j} \big(|A|^2 +|V| \big) e^{itH}f \Big\|_{L^2_{x,t}} \\
    &\quad\leq \sum_{j \in \mathbb{Z}} 2^{j/2} \Big( \sup_{x\in C_j} \big(|A|^2 +|V| \big) \Big)
        \Big( \int_{2^j}^{2^{j+1}} r^2 \int_{0}^{\infty}  \frac{1}{r^2} \int_{|x|= r} |u|^2 d\sigma_r(x) dt dr \Big)^{\frac{1}{2}}\\
    &\quad\leq \Big( \sum_{j \in \mathbb{Z}} 2^{2j} \sup_{x\in C_j} \big(|A|^2 +|V| \big) \Big)
        \Big( \sup_{R>0} \frac{1}{R^2} \int_{0}^{\infty} \int_{|x|= R} |u|^2 d\sigma_R(x)dt \Big)^{\frac{1}{2}}\\
    &\quad\leq C_{A,V,\epsilon} \big\| (-\Delta_{A})^\frac{1}{4} f \big\|_{L^2_x}.
    \end{aligned}
    $$
That is,
    $$
    \big\| |\nabla|^{\frac{1}{2}} e^{itH}f \big\|_{L^q_t L^r_x}
    \leq C_{n,r,q} \big\| |\nabla|^{1/2} f \big\|_{L^2_x} + C_{n,r,q,A,V,\epsilon} \big\| (-\Delta_{A})^\frac{1}{4} f \big\|_{L^2_x}.
    $$

First, consider the case $n=3$.
By \eqref{Cond_2'}, estimate \eqref{key_leq_ineq} in Proposition \ref{Key_Prop} holds for all $1<p\leq 6$.
(Here, $H= -\Delta_{A} +V$.)
Then by \eqref{key_geq_ineq} in Proposition \ref{Key_Prop}, we get
    \begin{equation}\label{commute_relation}
    \begin{aligned}
    \big\| H^{\frac{1}{4}} e^{itH}f \big\|_{L^{q}_{t}(\mathbb{R}; L^{r}_{x}(\mathbb{R}^3))}
    &\leq C \big\| |\nabla|^{\frac{1}{2}} e^{itH}f \big\|_{L^{q}_{t}(\mathbb{R}; L^{r}_{x}(\mathbb{R}^3))} \\
    &\leq C \big\| |\nabla|^{\frac{1}{2}} f \big\|_{L^2_x(\mathbb{R}^3)} + C \big\| (-\Delta_{A})^{\frac{1}{4}} f \big\|_{L^2_x(\mathbb{R}^3)} \\
    &\leq C \big\| H^\frac{1}{4} f \big\|_{L^2_x(\mathbb{R}^3)}
    \end{aligned}
    \end{equation}
for all admissible pairs $(r,q)$.
(It clearly includes the endpoint case $(n,r,q)=(3,6,2)$.)

Next, for the case $n\geq4$, we already know that
\eqref{key_leq_ineq} holds for $1<p<2n$ and that \eqref{key_geq_ineq} is valid for $\frac{4}{3}<p<4$ under the same conditions on $A$ and $V$
(see \cite[Theorem 1.2]{DFVV}).
Thus we can easily get the same bound as \eqref{commute_relation} for all admissible pairs $(r,q)$.

Since the operators $H^{\frac{1}{4}}$ and $e^{itH}$ commutes, we get
    $$
    \big\| e^{itH}f \big\|_{L^{q}_{t}(\mathbb{R}; L^{r}_{x}(\mathbb{R}^n))}
    \leq C \big\| f \big\|_{L^2_x(\mathbb{R}^n)}
    $$
from \eqref{commute_relation}, and this completes the proof.



\section*{Acknowledgment}

Y. Koh was partially supported by NRF Grant 2016R1D1A1B03932049 (Republic of Korea).
The authors thank the referee for careful reading of the manuscript and many invaluable comments.


\end{document}